\documentclass[11pt,twoside,reqno,centertags,draft]{amsart}
\usepackage{amsfonts}

  \usepackage{amsmath,amsthm,amsfonts,amssymb}
\begin{document}

\title{Multiple nodal solutions of nonlinear Choquard equations}
\date{}
\maketitle

\vspace{ -1\baselineskip}

{\small
\begin{center}
  {\sc  Zhihua Huang, Jianfu Yang  and Weilin Yu} \\
Department of Mathematics,
Jiangxi Normal University\\
Nanchang, Jiangxi 330022,
P.~R.~China\\
email:Z.Huang:zhhuang2016@126.com,  J.Yang: jfyang\_2000@yahoo.com and W. Yu: williamyu2065@163.com\\[10pt]
\end{center}
}

\renewcommand{\thefootnote}{}

\begin{quote}
{\bf Abstract.} In this paper, we consider the existence of multiple nodal solutions of the nonlinear Choquard equation
\begin{equation*}
\left\{
\begin{aligned}
&-\Delta u+u=(|x|^{-1}\ast|u|^p)|u|^{p-2}u && \text{in $\mathbb{R}^3$,}\\
&u\in H^1(\mathbb{R}^3),\\
\end{aligned}
\right.\leqno{(P)}
\end{equation*}
where $p\in (\frac{5}{2},5)$. We show that for any positive integer $k$, problem $(P)$ has at least a radially symmetrical solution changing sign exactly $k$-times.
\end{quote}

\newcommand{\N}{\mathbb{N}}
\newcommand{\R}{\mathbb{R}}
\newcommand{\Z}{\mathbb{Z}}

\newcommand{\cA}{{\mathcal A}}
\newcommand{\cB}{{\mathcal B}}
\newcommand{\cC}{{\mathcal C}}
\newcommand{\cD}{{\mathcal D}}
\newcommand{\cE}{{\mathcal E}}
\newcommand{\cF}{{\mathcal F}}
\newcommand{\cG}{{\mathcal G}}
\newcommand{\cH}{{\mathcal H}}
\newcommand{\cI}{{\mathcal I}}
\newcommand{\cJ}{{\mathcal J}}
\newcommand{\cK}{{\mathcal K}}
\newcommand{\cL}{{\mathcal L}}
\newcommand{\cM}{{\mathcal M}}
\newcommand{\cN}{{\mathcal N}}
\newcommand{\cO}{{\mathcal O}}
\newcommand{\cP}{{\mathcal P}}
\newcommand{\cQ}{{\mathcal Q}}
\newcommand{\cR}{{\mathcal R}}
\newcommand{\cS}{{\mathcal S}}
\newcommand{\cT}{{\mathcal T}}
\newcommand{\cU}{{\mathcal U}}
\newcommand{\cV}{{\mathcal V}}
\newcommand{\cW}{{\mathcal W}}
\newcommand{\cX}{{\mathcal X}}
\newcommand{\cY}{{\mathcal Y}}
\newcommand{\cZ}{{\mathcal Z}}

\newcommand{\abs}[1]{\lvert#1\rvert}
\newcommand{\xabs}[1]{\left\lvert#1\right\rvert}
\newcommand{\norm}[1]{\lVert#1\rVert}

\newcommand{\loc}{\mathrm{loc}}
\newcommand{\p}{\partial}
\newcommand{\h}{\hskip 5mm}
\newcommand{\ti}{\widetilde}
\newcommand{\D}{\Delta}
\newcommand{\e}{\epsilon}
\newcommand{\bs}{\backslash}
\newcommand{\ep}{\emptyset}
\newcommand{\su}{\subset}
\newcommand{\ds}{\displaystyle}
\newcommand{\ld}{\lambda}
\newcommand{\vp}{\varphi}
\newcommand{\wpp}{W_0^{1,\ p}(\Omega)}
\newcommand{\ino}{\int_\Omega}
\newcommand{\bo}{\overline{\Omega}}
\newcommand{\ccc}{\cC_0^1(\bo)}
\newcommand{\iii}{\opint_{D_1}D_i}

\theoremstyle{plain}
\newtheorem{Thm}{Theorem}[section]
\newtheorem{Lem}[Thm]{Lemma}
\newtheorem{Def}[Thm]{Definition}
\newtheorem{Cor}[Thm]{Corollary}
\newtheorem{Prop}[Thm]{Proposition}
\newtheorem{Rem}[Thm]{Remark}
\newtheorem{Ex}[Thm]{Example}

\numberwithin{equation}{section}
\newcommand{\meas}{\rm meas}
\newcommand{\ess}{\rm ess} \newcommand{\esssup}{\rm ess\,sup}
\newcommand{\essinf}{\rm ess\,inf} \newcommand{\spann}{\rm span}
\newcommand{\clos}{\rm clos} \newcommand{\opint}{\rm int}
\newcommand{\conv}{\rm conv} \newcommand{\dist}{\rm dist}
\newcommand{\id}{\rm id} \newcommand{\gen}{\rm gen}
\newcommand{\opdiv}{\rm div}

\vskip 0.2cm \arraycolsep1.5pt
\newtheorem{Lemma}{Lemma}[section]
\newtheorem{Theorem}{Theorem}[section]
\newtheorem{Definition}{Definition}[section]
\newtheorem{Proposition}{Proposition}[section]
\newtheorem{Remark}{Remark}[section]
\newtheorem{Corollary}{Corollary}[section]

\section {Introduction}

In this paper, we consider the existence of multiple nodal solutions for the nonlinear Choquard equation
\begin{equation*}
\left\{
\begin{aligned}
&-\Delta u+u=(|x|^{-1}\ast|u|^p)|u|^{p-2}u && \text{in $\mathbb{R}^3$,}\\
&u\in H^1(\mathbb{R}^3)\\
\end{aligned}
\right.\leqno{(P)}
\end{equation*}
where $p\in (\frac{5}{2},5)$.

In the case $p=2$, equation $(P)$ is the Choquard-Pekar equation introduced by Pekar in \cite{P}, see also Section 2.1 in \cite{DA}, to describe the quantum theory of a polaron at rest and proposed by Choquard \cite{L} in the study of a certain approximation to Hartree-Fock theory for one component plasma.  Further physical consideration of $(P)$, known as the Schr\"{o}dinger-Poisson equation, can be found in \cite{J, MPT} as a model of self-gravitating matter and in \cite{Len} as a non-relativistic model of boson stars.

Mathematically, it is early around 1980's, nonlinear Choquard equation $(P)$ was studied in \cite{L, Lions, Lions1, M} by the variational method, and recently, this problem and its generalization have been attractive in researches.  Existence and qualitative properties of solutions have been investigated in \cite{CCS, CS, CSS, GS, MS,MS1,MS2} and references therein.
In particular, the existence of nodal solutions for the Choquard equation is an appealing aspect, this aspect is investigated in \cite{CCS1,CSS,ClapS, GS} etc by the variational method, that is, by seeking for critical points of the associated functional. The energy functional associated to the Choquard equation $(P)$ is defined for each $u$ in  $H^1(\mathbb{R}^3)$ by
\begin{equation}\label{eq:1.1}
I(u)=\frac{1}{2}\int_{\mathbb{R}^3}(|\nabla u|^2+|u|^2)dx-\frac{1}{2p}\int_{\mathbb{R}^3}\int_{\mathbb{R}^3}\frac{|u(x)|^p|u(y)|^p}{|x-y|}\,dxdy.
\end{equation}
By the Hardy-Littlewood-Sobolev inequality, the functional $I$ is well defined on $H^1(\mathbb{R}^3)$ if $p\in (\frac{5}{2},5)$. Hence, critical points of $I(u)$ are weak solutions of problem $(P)$, and necessarily contained in the Nehari  manifold
\[
\mathcal{N}=\{u\in H^1(\mathbb{R}^3)\big|u\neq0,\langle I'(u),u\rangle=0\}.
\]
A standard way to find critical points of $I$ is to seek for minimizers of the functional $I$ constraint on the Nehari manifold $\mathcal{N}$. This idea was used in \cite{GS} in constructing a sign-changing solution for the Choquard equation in an odd Nehari manifold. Another way to construct a nodal solution is to find a critical point of $I$ in the Nehari set
\[
\mathcal{N}_0=\{u\in H^1(\mathbb{R}^3)\big|u^{\pm}\neq0,\langle I'(u),u^{\pm}\rangle=0\}.
\]
However, $\mathcal{N}_0$ is not a manifold. The argument then among other things, lies in showing that there is a minimizer of $I$ constraint on $\mathcal{N}_0$, and verifying that the minimizer is a critical point of $I$. Using this approach, a sign-changing solution is constructed in \cite{GS} for the Choquard equation, and in \cite{AS,WZ} for the nonlinear Schr\"{o}dinger-Poisson system and in \cite{BS, FN} for the Kirchhoff equation, further results can found in references therein.

In this paper, we intend to show that for every fixed integer $k$, there exists a radial solution of problem $(P)$ which changes sign exactly $k$ times. Particularly,
for $k=2$, there is a radially sign changing solution of problem $(P)$.

For every integer $k\geq0$, it was proved in \cite{BW} and \cite{CZ} independently that,  there is a pair of solutions $u^\pm_k$ having exact $k$ nodes of
\begin{equation}\label{eq:1.2}
\left\{
\begin{aligned}
&-\Delta u+V(|x|)u=f(|x|,u)\,\,{\rm in}\,\,\mathbb{R}^N,\\
&u\in H^1(\mathbb{R}^N).\\
\end{aligned}
\right.
\end{equation}
Such solutions of \eqref{eq:1.2} are obtained by gluing solutions of the equation in each annulus, including every ball and the complement of it. However,
this approach cannot be applied directly to problems with nonlocal terms, because nonlocal terms need
the global information of $u$. This difficulty was overcome by regarding the problem as a system of $k+1$ equations with $k+1$ unknown functions $u_i$, each $u_i$ is supported on only one annulus and vanishes at the complement of it. This argument relies on, among other things, constructing a functional $E_k$ and a Nehari type manifold $\mathcal{N}_k$, then finding a minimizer of $E_k$ constraint on $\mathcal{N}_k$. In this way, Kim and Seok \cite{KS} found infinitely many nodal solutions for
Schr\"{o}dinger-Poisson system, and then Deng et at \cite{DPS} treated Kirchhoff problems in $\mathbb{R}^3$ in a similar way. However, this argument can not be simply carried out to deal with the Choquard equation $(P)$, because in the proof of $\mathcal{N}_k$ being a manifold for problems considered in \cite{DPS} and \cite{KS}, a key ingredient used is that the related matrix is diagonally dominant at each point of  $\mathcal{N}_k$, but this is not the case for the Choquard equation $(P)$. In this paper, we find a way to show that the matrix associated to our Nehari type set $\mathcal{N}_k$ is nonsingular, the fact eventually allows us to verify that $\mathcal{N}_k$ is a manifold. This method might be possible to apply to analogous problems.

Our main result in this paper is stated as follows.

\begin{Theorem}\label{thm:1.1}  Suppose $\frac{5}{2}<p<5$. For every positive integer $k$, there exists a radial solution of  $(P)$, which changes sign exactly $k$-times.
\end{Theorem}

Theorem \ref{thm:1.1} will be proved by the variational method. We will define in Section 2 a functional $E=E(u_1,\cdots,u_{k+1})$ on  $\mathcal{H}_k=H_1\times \cdots \times H_{k+1}$, where $H_i$ are Hilbert spaces for $i=1,\cdots,k+1$. Then, we consider the variational problem
\[
E_{min} = \inf_{(u_1,\cdots,u_{k+1})\in\mathcal{N}_k}E(u_1,\cdots,u_{k+1}),
\]
where
\[
\mathcal{N}_k=\big\{(u_{1},\cdots,u_{k+1})\in \mathcal{H}_k\ \big|\ u_{i}\neq 0,\partial_{u_{i}}E(u_{1},\cdots,u_{k+1})u_{i}=0\,\,{\rm for\,\, each}\,\,i.\big\}
\]
is a Nehari type set. It will be shown that each component of a minimizer $(u_1,\cdots,u_{k+1})\in \mathcal{N}_k$ of $E_{min}$ is a solution of the problem on decomposed regions. Hence, it is necessary to verify
that $\mathcal{N}_k$ is a manifold, where a difficulty arises. Nodal solutions of problem $(P)$ will be constructed by gluing each component of a minimizer $(u_1,\cdots,u_{k+1})\in \mathcal{N}_k$ of $E_{min}$ together.

This paper is organized as follows. In Section 2, we present variational framework to deal with problem $(P)$ and find a minimizer of the related minimization problem. Nodal solutions of problem $(P)$ will be constructed in Section 3.

\section {Preliminaries}

\bigskip

In this section, we present the variational framework and modify the energy functional $I$ to a functional corresponding to a system of $(k+1)$-equations.
For each $k\in\mathbb{N}_+ $, we define
\[
\mathbf{\Gamma}_k:=\big\{\mathbf{r}_k=(r_1,\cdots,r_k)\in \mathbb{R}^k\ \big| \ 0=r_0<r_1<\cdots<r_k<r_{k+1}=\infty\big\},
\]
and denote
\[
B_1=B_1^{\mathbf{r}_k}=\{x\in\mathbb{R}^3 : 0\leq|x|<r_1\}
\]
and
\[
B_i=B_i^{\mathbf{r}_k}=\{x\in\mathbb{R}^3: r_{i-1}<|x|<r_i\}
\]
for $i=2,\cdots,k+1$. Therefore, $B_1$ is a ball, $B_2,\cdots,B_k$ are annuli and $B_{k+1}$ is the complement of a ball.  Fix $\mathbf{r}_k=(r_1,\cdots,r_k)\in\mathbf{\Gamma}_k$ and thereby a family of $\{B_i\}_{i=1}^{k+1}$, we denote
\[
H_i:=\big\{u\in H_0^1(B_i)\ \big|\ u(x)=u(|x|),u(x)=0\ {\rm if}\ x\notin B_i\big\}
\]
for $i=1,\cdots,k+1$. It can be verified that $ H_i$ is a Hilbert space with the norm
\[
\|u\|_i^2=\int_{B_i}(|\nabla u|^2+u^2)dx.
\]
Let $\mathcal{H}_k=H_1\times \cdots \times H_{k+1}$. We define the functional $E:\mathcal{H}_k\rightarrow \mathbb{R}$ by
\begin{equation}\label{eq:2.1}
\begin{aligned}
E(u_1,\cdots,u_{k+1}):=&\frac{1}{2}\sum_{i=1}^{k+1}\|u_i\|_i^2
            -\frac{1}{2p}\sum_{i=1}^{k+1}\int_{B_i}\int_{B_i}\frac{|u_i(x)|^p|u_i(y)|^p}{|x-y|}dxdy\\
            &-\frac{1}{2p}\sum_{j\neq i}^{k+1}\int_{B_i}\int_{B_j}\frac{|u_i(x)|^p|u_j(y)|^p}{|x-y|}dxdy,
\end{aligned}
\end{equation}
where $u_i\in H_i$,  $i=1,\cdots,k+1$. It is obvious that
\[
E(u_1,\cdots,u_{k+1})=I(\sum_{i=1}^{k+1}u_i).
\]
Moreover, if $(u_1,\cdots,u_{k+1})\in \mathcal{H}_k$ is a critical point of $E$, then each component $u_i$ satisfies
\begin{equation*}
\left\{
\begin{aligned}
&-\Delta u_i+u_i=(|x|^{-1}\ast|\sum_{i=1}^{k+1}u_i|^p)|u_i|^{p-2}u_i ,&& x\in B_i\\
&u_i=0 ,&&x\notin B_i\\
\end{aligned}
\right.\leqno{(P_i)}
\end{equation*}
Nodal solutions of problem $(P)$ will be constructed by gluing solutions of problem $(P_i), i=1,\cdots,k+1$ up.
In order to find critical points of $E$ with nonzero component, we consider the minimization problem
\begin{equation}\label{eq:2.2}
E_{min} = \inf_{(u_1,\cdots,u_{k+1})\in\mathcal{N}_k}E(u_1,\cdots,u_{k+1})
\end{equation}
constrained on the Nehari type set
\[
\begin{split}
\mathcal{N}_k = &\big\{(u_{1},\cdots,u_{k+1})\in \mathcal{H}_k\,\,\big|u_{i}\neq 0,\partial_{u_{i}}E(u_{1},\cdots,u_{k+1})u_{i}=0,\\
&\,\,\, i=1,\cdots,k+1\big\},\\
\end{split}
\]
where
\[
\begin{split}
&\partial_{u_{i}}E(u_{1},\cdots,u_{k+1})u_{i} \\
= &\|u_{i}\|_{i}^{2}-\int_{B_i}\int_{B_i}\frac{|u_i(x)|^p|u_i(y)|^p}{|x-y|}\,dxdy -\sum_{j\neq i}^{k+1}\int_{B_i}\int_{B_j}\frac{|u_i(x)|^p|u_j(y)|^p}{|x-y|}\,dxdy.\\
\end{split}
\]
It is necessary to show that the set $\mathcal{N}_k$ is nonempty, and then $E_{min}$ is well defined. We know that a minimizer $u$ of $E_{min}$ is a critical point of $E_{min}$ constrained on $\mathcal{N}_k$ if $\mathcal{N}_k$ is a manifold in $\mathcal{H}_k$, hence, each component $u$ is possibly a solution of problem $(P_i)$. In this section, we will prove these facts, and find a solution of problem $(P_i)$ for each $i$. We commence with proving the set $\mathcal{N}_k$ is nonempty.

\begin{Lemma}\label{lem:2.1}
Assume that $p\in(\frac{5}{2},5)$. For $(u_{1},\cdots,u_{k+1})\in \mathcal{H}_k$ with $u_{i}\neq0$ for $i=1,\cdots,k+1$, there is a unique $(k+1)$-tuple $(t_{1},\cdots,t_{k+1})$ of positive numbers such that $(t_{1}u_{1},\cdots,t_{k+1}u_{k+1}) \in \mathcal{N}_k$.
\end{Lemma}
\begin{proof}
Fix $(u_{1},\cdots,u_{k+1})\in \mathcal{H}_k$ with $u_{i}\neq0$,  $i=1,\cdots,k+1$.  Then, $(t_{1}u_{1},\cdots,t_{k+1}u_{k+1})\in \mathcal{N}_k$ for some $(t_{1},\cdots,t_{k+1})\in (\mathbb{R}_{>0})^{k+1}$ if and only if
\begin{equation}\label{eq:2.3}
\begin{aligned}
t_i^2\|u_{i}\|_{i}^{2}&-t_i^{2p}\int_{B_i}\int_{B_i}\frac{|u_i(x)|^p|u_i(y)|^p}{|x-y|}\,dxdy\\
&-\sum_{j\neq i}^{k+1}\int_{B_i}\int_{B_j}\frac{t_i^pt_j^p|u_i(x)|^p|u_j(y)|^p}{|x-y|}\,dxdy=0
\end{aligned}
\end{equation}
for $i=1,\cdots,k+1$. Hence, the problem is reduced to verify that there is only one solution $(t_{1},\cdots,t_{k+1})$ of system \eqref{eq:2.3} with $t_i>0$, for each $i=1,\cdots,k+1$.
To this end, we introduce a parameter $0\leq\mu\leq1$, and consider the solvability of the following system of $(k+1)$ equations
\begin{equation}\label{eq:2.4}
\begin{aligned}
&G_i(t_1,\cdots,t_{k+1}):=t_i^2\|u_{i}\|_{i}^{2}-t_i^{2p}\int_{B_i}\int_{B_i}\frac{|u_i(x)|^p|u_i(y)|^p}{|x-y|}\,dxdy\\
&-\mu\sum_{j\neq i}^{k+1}\int_{B_i}\int_{B_j}\frac{t_i^pt_j^p|u_i(x)|^p|u_j(y)|^p}{|x-y|}\,dxdy=0,\ i=1,\cdots,k+1.
\end{aligned}
\end{equation}

Let
\begin{equation}\label{eq:2.5}
\mathcal{Z}=\big\{\mu\ \big|\ 0\leq\mu\leq1 \ {\rm and} \ \eqref{eq:2.4}\  {\rm is}\ {\rm uniquely}\ {\rm solvable}\ {\rm in}\ (\mathbb{R}_{>0})^{k+1}\big\}.
\end{equation}
Apparently, $0\in \mathcal{Z}$, so the set $\mathcal{Z}$ is nonempty in $[0,1]$.  We claim that $\mathcal{Z}= [0,1]$, which implies the result.
To prove the claim, it is sufficient to show that $\mathcal{Z}$ is both open and closed in $[0,1]$.

We first prove that the set $\mathcal{Z}$ is open in $[0,1]$.

Suppose that $\mu_0\in\mathcal{Z}$ and $(\bar{t}_{1},\cdots,\bar{t}_{k+1})\in(\mathbb{R}_{>0})^{k+1}$ is the unique solution of \eqref{eq:2.4} with $\mu=\mu_0$. In order to apply the implicit function theorem at $\mu_0$, we calculate the matrix
\begin{equation}\label{eq:2.5a}
M=(M_{ij})=(\partial_{t_j}G_i)_{i,j=1,\cdots,k+1}.
\end{equation}
Each component of the matrix $M$ is then given by
\[
\begin{aligned}
M_{ii}&=2\bar{t}_i\|u_{i}\|_{i}^{2}-2p\bar{t}^{2p-1}_i\int_{B_i}\int_{B_i}\frac{|u_i(x)|^p|u_i(y)|^p}{|x-y|}\,dxdy\\
&\quad-\mu_0p\bar{t}^{p-1}_i\sum_{j\neq i}^{k+1}\int_{B_i}\int_{B_j}\frac{\bar{t}_j^p|u_i(x)|^p|u_j(y)|^p}{|x-y|}\,dxdy\\
&=(2-p)\bar{t}_i\|u_{i}\|_{i}^{2}-p\bar{t}^{2p-1}_i\int_{B_i}\int_{B_i}\frac{|u_i(x)|^p|u_i(y)|^p}{|x-y|}\,dxdy
\end{aligned}
\]
for $i=1,\cdots,k+1$, where we have used \eqref{eq:2.4}, and
\[
M_{ij}=-\mu_0p\bar{t}^{p}_i\bar{t}^{p-1}_j\int_{B_i}\int_{B_j}\frac{|u_i(x)|^p|u_j(y)|^p}{|x-y|}dxdy
\]
for $i\neq j$, $i,j=1,\cdots,k+1$. Therefore,
\begin{equation}\label{eq:2.7}
{\rm det}\ M=\frac{(-1)^{k+1}}{\bar{t}_1\cdots \bar{t}_{k+1}}\ {\rm det}\ \widetilde{M},
\end{equation}
where components of the matrix $\widetilde{M}=(\widetilde{M}_{ij})$ are given by
\[
\widetilde M_{ii}=(p-2)\bar{t}_i^2\|u_{i}\|_{i}^{2}+p\bar{t}^{2p}_i\int_{B_i}\int_{B_i}\frac{|u_i(x)|^p|u_i(y)|^p}{|x-y|}\,dxdy
\]
for $i=1,\cdots,k+1$, and
\[
\widetilde M_{ij}=\mu_0p\bar{t}^{p}_i\bar{t}^{p}_j\int_{B_i}\int_{B_j}\frac{|u_i(x)|^p|u_j(y)|^p}{|x-y|}\,dxdy,\ {\rm for}\ i\neq j,\  i,j=1,\cdots,k+1.
\]
By Lemma \ref{lem:A} in the appendix, we obtain
\[
{\rm det}\ M\neq0.
\]
Hence,  the implicit function theorem implies that there are an open neighborhood $U_0$ of $\mu_0$ and a neighborhood $ A_0 \subset (\mathbb{R}_{>0})^{k+1}$ of $(\bar{t}_{1},\cdots,\bar{t}_{k+1})$ such that system \eqref{eq:2.4} is uniquely solvable in $U_0\times A_0$.

Now we show \eqref{eq:2.4} is uniquely solvable in $ U_0\times  (\mathbb{R}_{>0})^{k+1}$, this means $U_0\subset \mathcal{Z}$, and $\mathcal{Z}$ is open.
Suppose, on the contrary, that there is $\mu_1 \in U_0 $ such that there exists the second solution  $(\tilde{t}_{1},\cdots,\tilde{t}_{k+1})\in (\mathbb{R}_{>0})^{k+1}\setminus A_0$ of  \eqref{eq:2.4}. By the implicit function theorem, we can find a solution curve $(\mu,(\tilde{t}_{1}(\mu),\cdots,\tilde{t}_{k+1}(\mu)))$ in $(\mu_1-\varepsilon,\mu_1+\varepsilon)\times\big((\mathbb{R}_{>0})^{k+1}\setminus A_0\big)$.
If $\mu_0<\mu_1$, we extend this curve as much as possible. Since it cannot be defined at $\mu_0$ and enter into $U_0\times A_0$, there should have a point $\mu_2\in [\mu_0,\mu_1)$ such that $(t_1(\mu),\cdots,t_{k+1}(\mu))$ being defined in $(\mu_2,\mu_1]$ and blowing up as $\mu\rightarrow \mu_2^+ $. However, this is impossible, since if $(t_1,\cdots,t_{k+1})$ has sufficiently large norm, the left-hand side of  \eqref{eq:2.4} is strictly negative for at least one $ i $. This gives a contradiction. Thus, $U_0\subset \mathcal{Z}$. The case $\mu_0 >\mu_1 $ can be proved in the same way.

Next, we show that the set $\mathcal{Z}$ is closed in $[0,1]$.

Let $\{\mu_n\}$ be a sequence in $\mathcal{Z}$ converging to $\mu_0\in[0,1]$ and $(t_{1}^n,\cdots,t_{k+1}^n)\in(\mathbb{R}_{>0})^{k+1}$ be the solution of \eqref{eq:2.4} for $\mu_n$. By the preceding argument, we see that the sequence $(t_{1}^n,\cdots,t_{k+1}^n)$ is bounded above. Thus we may assume that $(t_{1}^n,\cdots,t_{k+1}^n)$ converges to a solution $(t_{1}^0,\cdots,t_{k+1}^0)\in(\mathbb{R}_{\geq0})^{k+1}$ of \eqref{eq:2.4} for $\mu_0$. Let $v^n=t_1^n u_1+\cdots+t^n_{k+1}u_{k+1}$. Since $\{v_n\}$ is uniformly bounded in $\mathcal{H}_k$, by \eqref{eq:2.4} and the Hardy-Littlewood-Sobolev inequality, we derive
\begin{equation}\label{eq:2.8}
\begin{aligned}
(t_i^n)^2\|u_{i}\|_{i}^{2}&=(t_i^n)^{2p}\int_{B_i}\int_{B_i}\frac{|u_i(x)|^p|u_i(y)|^p}{|x-y|}\,dxdy\\
&\quad+\mu_n\sum_{j\neq i}^{k+1}\int_{B_i}\int_{B_j}\frac{(t_i^n)^p(t_j^n)^p|u_i(x)|^p|u_j(y)|^p}{|x-y|}\,dxdy\\
&\leq(t_i^n)^{2p}\int_{B_i}\int_{B_i}\frac{|u_i(x)|^p|u_i(y)|^p}{|x-y|}dxdy\\
&\quad+\sum_{j\neq i}^{k+1}\int_{B_i}\int_{B_j}\frac{(t_i^n)^p(t_j^n)^p|u_i(x)|^p|u_j(y)|^p}{|x-y|}\,dxdy\\
&=\int_{B_i}\int_{\mathbb{R}^3}\frac{|t^n_iu_i(x)|^p|v^n(y)|^p}{|x-y|}dxdy\\
&\leq C_1(t_i^n)^p\|u_{i}\|_{\frac{6p}{5}}^{p}\|v^n\|_{\frac{6p}{5}}^{p}\leq C_2(t_i^n)^p\|u_{i}\|_{i}^{p}.
\end{aligned}
\end{equation}
This implies that $0<C_i<t_i^n$ holds uniformly in $n$. As a result, $t_i^0\geq C_i>0$ for $i=1,\cdots,k+1$, that is, $(t_{1}^0,\cdots,t_{k+1}^0)\in(\mathbb{R}_{>0})^{k+1}$.
By the implicit function theorem again,  $(t_{1}^0,\cdots,t_{k+1}^0)$ is the unique solution of \eqref{eq:2.4} in $(\mathbb{R}_{>0})^{k+1}$. Hence, $\mathcal{Z}$ is closed.
The conclusion of Lemma \ref{lem:2.1} then follows.

\end{proof}

\begin{Lemma}\label{lem:2.2}
For any $\frac{5}{2}<p<5$, $\mathcal{N}_k$ is a differentiable manifold in $\mathcal{H}_k$. Moreover, all critical points of the restriction $E\big|_{\mathcal{N}_k}$ of $E$ to $\mathcal{N}_k$ are critical points of $E$ with no zero component.
\end{Lemma}
\begin{proof} We show that $\mathcal{N}_k$ is a manifold first. We may write
$$
\mathcal{N}_k=\{(u_{1},\cdots,u_{k+1})\in \mathcal{H}_k\ \big|\ u_{i}\neq0,\mathbf{F}(u_{1},\cdots,u_{k+1})=\mathbf{0}\},
$$
where $\mathbf{F}=(F_1,\cdots,F_{k+1}):\mathcal{H}_k\rightarrow \mathbb{R}^{k+1}$ is given by
\begin{equation}\label{eq:2.9}
\begin{aligned}
F_i(u_1,\cdots,u_{k+1})=&\|u_{i}\|_{i}^{2}-\int_{B_i}\int_{B_i}\frac{|u_i(x)|^p|u_i(y)|^p}{|x-y|}\,dxdy\\
&-\sum_{j\neq i}^{k+1}\int_{B_i}\int_{B_j}\frac{|u_i(x)|^p|u_j(y)|^p}{|x-y|}\,dxdy
\end{aligned}
\end{equation}
for $i=1,\cdots,k+1$.

In order to prove that $\mathcal{N}_k$ is a differentiable manifold in $\mathcal{H}_k$, it suffices to check that the matrix
\begin{equation}\label{eq:2.10}
N:=(N_{ij})=\big((\partial_{u_i}F_j(u_1,\cdots,u_{k+1}), u_i)\big)_{i,j=1,\cdots,k+1}
\end{equation}
is nonsingular at each point $(u_1,\cdots,u_{k+1})\in \mathcal{N}_k$, since it implies that $0$ is a regular value of $\mathbf{F}$. By direct computation, we have
\[
\begin{aligned}
N_{ii}&=2\|u_{i}\|_{i}^{2}-2p\int_{B_i}\int_{B_i}\frac{|u_i(x)|^p|u_i(y)|^p}{|x-y|}\,dxdy\\
&\quad-p\sum_{j\neq i}^{k+1}\int_{B_i}\int_{B_j}\frac{|u_i(x)|^p|u_j(y)|^p}{|x-y|}\,dxdy\\
&=(2-p)\|u_{i}\|_{i}^{2}-p\int_{B_i}\int_{B_i}\frac{|u_i(x)|^p|u_i(y)|^p}{|x-y|}\,dxdy,
\end{aligned}
\]
for $i=1,\cdots,k+1$, and
\[
N_{ij}=-p\int_{B_i}\int_{B_j}\frac{|u_i(x)|^p|u_j(y)|^p}{|x-y|}dxdy,
\]
for $i\neq j$ and $i,j=1,\cdots,k+1$. By Lemma \ref{lem:A}, we may verify as the proof  of Lemma \ref{lem:2.1} that $det\ N\neq0$ at each point of $\mathcal{N}_k$. So $\mathcal{N}_k$ is a differentiable manifold in $\mathcal{H}_k$.

Next, we verify that any critical point $(u_1,\cdots,u_{k+1})$ of $E\big|_{\mathcal{N}_k}$ is a critical point of $E$. Indeed, if $(u_1,\cdots,u_{k+1})$ is a critical point of $E\big|_{\mathcal{N}_k}$, then there are Lagrange multipliers $\lambda_1,\cdots,\lambda_{k+1}$ such that
\begin{equation}\label{eq:2.11}
\lambda_1F'_1(u_1,\cdots,u_{k+1})+\cdots+\lambda_{k+1}F'_{k+1}(u_1,\cdots,u_{k+1})=E'(u_1,\cdots,u_{k+1}).
\end{equation}
The values of the operator identity \eqref{eq:2.11} at points
\[
(u_1,0,\cdots,0), (0,u_2,0,\cdots,0),\cdots, (0,\cdots,0,u_{k+1})
\]
form a system
\[
N\begin{pmatrix} \lambda_1  \\ \vdots \\ \lambda_{k+1} \end{pmatrix}=\begin{pmatrix}\ 0\  \\ \vdots \\ \ 0\ \end{pmatrix}.
\]
Since the matrix $N$ is nonsingular at each point of $\mathcal{N}_k$,  $\lambda_1,\cdots,\lambda_{k+1}$ are all zero and $(u_1,\cdots,u_{k+1})$ is a critical point of $E$.

Finally, for any $(u_1,\cdots,u_{k+1})\in{\mathcal{N}_k}$, we may derive as inequality \eqref{eq:2.8} that each $u_i$ is bounded away from zero. Thus, critical points of $E$ in ${\mathcal{N}_k}$ cannot have any zero component. The proof is complete.
\end{proof}

For a fixed $(u_1,\cdots,u_{k+1})\in \mathcal{H}_k$ with nonzero component,  by Lemma \ref{lem:2.1} there exists a unique vector $(t_{1},\cdots,t_{k+1})$  such that $(t_1u_1,\cdots,t_{k+1}u_{k+1})\in \mathcal{N}_k$. The vector $(t_{1},\cdots,t_{k+1})$ has the following property.

\begin{Lemma}\label{lem:2.3}
 The vector $(t_{1},\cdots,t_{k+1})$ is the unique maximum point of the function $\phi:(\mathbb{R}_{>0})^{k+1}\rightarrow \mathbb{R}$ defined as
 $$
 \phi(c_1,\cdots,c_{k+1})=E(c_{1}u_{1},\cdots,c_{k+1}u_{k+1}).
 $$
\end{Lemma}
\begin{proof} By Lemma \ref{lem:2.1}, we know that $(t_{1},\cdots,t_{k+1})$ is the unique critical point of $\phi$ in $(\mathbb{R}_{>0})^{k+1}$.

Since $p\in(\frac{5}{2},5)$, it is observed that $\phi(c_{1},\cdots,c_{k+1})\rightarrow-\infty$ uniformly as $|(c_{1},\cdots,c_{k+1})|\rightarrow+\infty$, so it is sufficient to check that a maximum point cannot be achieved on the boundary of $(\mathbb{R}_{>0})^{k+1}$.
Choose $(c^0_{1},\cdots,c^0_{k+1})\in\partial(\mathbb{R}_{>0})^{k+1}$, without loss of generality, we may assume that $c_1^0=0$. Since
\[
\begin{aligned}
\phi(t,c^0_{2},\cdots,c^0_{k+1})&=E(tu_1,c^0_{2}u_2,\cdots,c^0_{k+1}u_{k+1})\\
&=\frac{t^2}{2}\|u_1\|_1^2-\frac{t^{2p}}{2p}\int_{B_1}\int_{B_1}\frac{|u_1(x)|^p|u_1(y)|^p}{|x-y|}\,dxdy\\
&\quad-\frac{t^{p}}{p}\sum_{i=2}^{k+1}\int_{B_1}\int_{B_i}\frac{|u_1(x)|^p|c_i^0u_i(y)|^p}{|x-y|}\,dxdy
+\frac{1}{2}\sum_{i=2}^{k+1}\|c_i^0u_i\|_i^2\\
&\quad-\frac{1}{2p}\sum_{i,j=2}^{k+1}\int_{B_i}\int_{B_j}\frac{|c_i^0u_i(x)|^p|c_i^0u_j(y)|^p}{|x-y|}\,dxdy
\end{aligned}
\]
is increasing with respect to $t$ if $t$ is small enough,  $(0,c^0_{2},\cdots,c^0_{k+1})$ is not a maximum point of $\phi$ in $(\mathbb{R}_{>0})^{k+1}$.
The assertion follows.
\end{proof}

\medskip
Finally, we have the following existence result for problem $(P_i)$.
\begin{Lemma}\label{lem:2.4}
For any $\frac{5}{2}<p<5$ and fixed $\mathbf{r}_k=(r_1,\cdots,r_k)\in\mathbf{\Gamma}_k$, there is a minimizer $(w_1,\cdots,w_{k+1})$ of $E\big|_{\mathcal{N}_k}$ such that each $(-1)^{i+1}w_i$ is positive on $B_i$ for $i=1,\cdots,k+1$. Moreover, $(w_1,\cdots,w_{k+1})$ satisfies $(P_i)$.
\end{Lemma}
\begin{proof}
By the Hardy-Littlewood-Sobolev inequality and Sobolev embedding theorem, we deduce for $(u_1,\cdots,u_{k+1})\in \mathcal{N}_k$ that
\[
\begin{split}
\|u_i\|_i^2&=\int_{\mathbb{R}^3}\int_{B_i}\frac{|u(x)|^p|u_i(y)|^p}{|x-y|}\,dxdy\\
&\leq C\|u_i\|_{\frac{6p}{5}}^p\|u\|_{\frac{6p}{5}}^p\\
&\leq C\|u\|^p\|u_i\|_i^p\\
&\leq C\|u_i\|_i^p.
\end{split}
\]
Hence, there exists a constant $\alpha_i>0$ such that $\|u_i\|_i\geq\alpha_i>0,\, i=1,\cdots,k+1$. If $(u_1,\cdots,u_{k+1})\in \mathcal{N}_k$, there holds
\begin{equation}\label{eq:2.12}
E(u_1,\cdots,u_{k+1})=(\frac{1}{2}-\frac{1}{2p})\sum_{i=1}^{k+1}\|u_i\|_i^2\geq\alpha>0
\end{equation}
for some $\alpha>0$. This implies that any minimizing sequence $\{(u_{1}^{n},\cdots,u_{k+1}^{n})\}$ of $E\big|_{\mathcal{N}_k}$ is bounded in $\mathcal{H}_k$.
We may assume that the minimizing sequence  $(u_{1}^{n},\cdots,u_{k+1}^{n})$ weakly converges to an element $(u_{1}^{0},\cdots,u_{k+1}^{0})$ in $\mathcal{H}_k$.

We claim that $u_{i}^{0}\neq0$ for each $i=1,\cdots,k+1$. Indeed, if $(u_{1}^{n},\cdots,u_{k+1}^{n})$ strongly converges to $(u_{1}^{0},\cdots,u_{k+1}^{0})$ in $\mathcal{H}_k$, we may show in the same way as the proof of \eqref{eq:2.8} that
$\|u_i^n\|_i^2\leq C\|u_i^n\|_i^p$ for each $i$, In other word, $\|u_i^n\|_i\geq\mu_i>0$, thereby  $\|u_i^0\|_i\geq\mu_i^0>0$ for $i=1,\cdots,k+1$.

Suppose now that $(u_{1}^{n},\cdots,u_{k+1}^{n})\nrightarrow(u_{1}^{0},\cdots,u_{k+1}^{0})$ strongly in $\mathcal{H}_k$ as $n\rightarrow\infty$. That is, $\|u_i^0\|_i<\liminf_{n\rightarrow\infty}\|u_i^n\|_i$ for at least one $i\in\{1,\cdots,k+1\}$. Again, we have $u_{i}^{0}\neq0$ for each $i=1,\cdots,k+1$.
Indeed, since $(u_{1}^{n},\cdots,u_{k+1}^{n})\in\mathcal{N}_k$,
\[
\|u_i^n\|_i^2=\int_{\mathbb{R}^3}\int_{B_i}\frac{|u^n(x)|^p|u_i^n(y)|^p}{|x-y|}\,dxdy
\]
and the inclusion  $H^1_r(\mathbb{R}^3)\hookrightarrow L^q(\mathbb{R}^3)$ is compact for $2<q<6$,
\begin{equation}\label{eq:2.12a}
\int_{\mathbb{R}^3}\int_{B_i}\frac{|u^n(x)|^p|u_i^n(y)|^p}{|x-y|}\,dxdy\to\int_{\mathbb{R}^3}\int_{B_i}\frac{|u^0(x)|^p|u_i^0(y)|^p}{|x-y|}\,dxdy
\end{equation}
as $n\to\infty$, we obtain
\[
\begin{split}
\|u_i^0\|_i^2&\leq\liminf_{n\rightarrow\infty}\|u_i^n\|_i^2\leq\lim_{n\rightarrow\infty}\int_{\mathbb{R}^3}\int_{B_i}\frac{|u^n(x)|^p|u_i^n(y)|^p}{|x-y|}dxdy\\
&=\int_{\mathbb{R}^3}\int_{B_i}\frac{|u^0(x)|^p|u_i^0(y)|^p}{|x-y|}\,dxdy\leq C\|u_i^0\|_i^p,\\
\end{split}
\]
implying that there exists a constant $\mu_0>0$ such that $\|u_i^0\|_i\geq\mu_0>0$.

Since each component of $(u_{1}^{0},\cdots,u_{k+1}^{0})$ is nonzero, by Lemma \ref{lem:2.1}, one can find $(t_{1}^{0},\cdots,t_{k+1}^{0})\in(\mathbb{R}_{>0})^{k+1}$ and $(t_{1}^{0},\cdots,t_{k+1}^{0})\neq(1,\cdots,1)$ such that $(t_{1}^{0}u_{1}^{0},\cdots,t_{k+1}^{0}u_{k+1}^{0})\in\mathcal{N}_k$. But, in this case, by \eqref{eq:2.12a}
and Lemma \ref{lem:2.3} we derive that
\[
\begin{aligned}
&\inf_{(u_1,\cdots,u_{k+1})\in \mathcal{N}_k}E(u_1,\cdots,u_{k+1})\\
&\leq E(t_1^0u_1^0,\cdots,t_{k+1}^{0}u_{k+1}^{0})\\
&<\liminf_{n\rightarrow\infty}\{\frac{1}{2}\sum_{i=1}^{k+1}(t_i^0)^2\|u_i^n\|_i^2
-\frac{1}{2p}\sum_{i=1}^{k+1}(t_i^0)^{2p}\int_{B_i}\int_{B_i}\frac{|u_i^n(x)|^p|u_i^n(y)|^p}{|x-y|}\,dxdy\\
&\quad-\frac{1}{2p}\sum_{j\neq i}^{k+1}\int_{B_i}\int_{B_j}\frac{(t_i^0)^p(t_j^0)^p|u_i^n(x)|^p|u_j^n(y)|^p}{|x-y|}\,dxdy\}\\
&\leq\liminf_{n\rightarrow\infty}E(u_1^n,\cdots,u_{k+1}^n)\\
&=\inf_{(u_1,\cdots,u_{k+1})\in \mathcal{N}_k}E(u_1,\cdots,u_{k+1}),
\end{aligned}
\]
 which is a contradiction. Therefore, $(u_1^n,\cdots,u_{k+1}^n)$ strongly converges to $(u_1^0,\cdots,u_{k+1}^0)$ in $\mathcal{H}_k$ and $(u_1^0,\cdots,u_{k+1}^0)\in\mathcal{N}_k$ is a minimizer of $E\big|_{\mathcal{N}_k}$.

Furthermore, we may check that
\[
(w_1,\cdots,w_{k+1}):=(|u_1^0|,-|u_2^0|,\cdots,(-1)^k|u_{k+1}^0|)
\]
is also in $\mathcal{N}_k$ and is a minimizer of $E\big|_{\mathcal{N}_k}$. Hence, it is a critical point of $E\big|_{\mathcal{N}_k}$.  By Lemma \ref{lem:2.2}, it is also a critical point of $E$
and satisfies $(P_i)$. The strong maximum principle yields that each $(-1)^{i+1}w_i$ is positive in $B_i$. The assertion follows.

\end{proof}

\bigskip

\section {Existence of sign-changing radial solutions}

\bigskip

 It is known that for any ${\bf r}_k=(r_1,\cdot \cdot \cdot,r_k) \in {\bf \Gamma}_k$, there is a solution $w^{\mathbf{r}_k}=(w_1^{\mathbf{r}_k},\cdots,w_{k+1}^{\mathbf{r}_k})$ of $(P_i)$ which consists of sign changing components. We will find a ${\bf \bar r}_k=(\bar r_1,\cdot \cdot \cdot,\bar r_k) \in {\bf \Gamma}_k$ such that $w^{{\bf \bar r}_k}=(w^{{\bf \bar r}_k}_1,\cdot \cdot \cdot,w^{{\bf \bar r}_k}_{k+1})$ is a solution of $(P_i)$ which is
characterized as a least energy solution among all elements in ${\bf \Gamma}_k$ with nonzero components. Using this solution as a building block, we will construct a radial solution of $(P)$ that changes sign exactly $k$ times.  Denote by $B_i^{{\bf r}_k}$ the nodal domain and by $E^{{\bf r}_k}$ the functional related to ${\bf r}_k$. Note that $w^{{\bf  r}_k}_i$ is $\mathcal{C}^2(B_i^{{\bf r}_k})$ for each $i$ by standard elliptic regularity results. Hence, it is enough to match the first derivative with respect to the radial variable, of adjacent components $w^{{\bf r}_k}_i$ and $w^{{\bf r}_k}_{i+1}$ at the point $r_i$ to ensure the existence of a solution of equation $(P)$ with $k$ times sign changing.

In order to find a least energy radial solution of $(P_i)$ among elements in ${\bf \Gamma}_k$ with nonzero components, we need to estimate the energy of the solution $(w^{{\bf r}_k}_1,\cdot \cdot \cdot,w^{{\bf r}_k}_{k+1})$
of $(P_i)$. To this end, we first define the function  $\psi:\mathbf{\Gamma}_k\rightarrow \mathbb{R}$ by
\begin{equation}\label{eq:3.1}
\begin{split}
\psi(\mathbf{r}_k)&=\psi(r_1,\cdots,r_k)=E^{\mathbf{r}_k}(w_1^{\mathbf{r}_k},\cdots,w_{k+1}^{\mathbf{r}_k})\\
&=\inf_{(u_1^{\mathbf{r}_k},\cdots,u_{k+1}^{\mathbf{r}_k})\in \mathcal{N}_k^{\mathbf{r}_k}}E^{\mathbf{r}_k}(u_1^{\mathbf{r}_k},\cdots,u_{k+1}^{\mathbf{r}_k})\\
\end{split}
\end{equation}
\begin{Lemma}\label{lem:3.1} Suppose $\frac{5}{2}<p<5$. For any positive integer $k$, let $\mathbf{r}_k=(r_1,\cdots,r_k)\in\mathbf{\Gamma}_k$. Then,

$(i)$ if $r_i-r_{i-1}\rightarrow0$ for some $i\in\{1,\cdots,k\}$, then $\psi(\mathbf{r}_k)\rightarrow+\infty$;

$(ii)$ if $r_k\rightarrow\infty$, then $\psi(\mathbf{r}_k)\rightarrow+\infty$;

$(iii)$ $\psi$ is continuous in $\mathbf{\Gamma}_k$.

In particular, there is a ${{\bf\bar r}_k}=(\bar r_1,\cdot \cdot \cdot,\bar r_k)\in {\bf\Gamma}_k$  such that
\[
\psi({\bf\bar r}_k) = \inf_{{\bf r}_k\in {\bf\Gamma}_k}\psi({\bf r}_k).
\]
\end{Lemma}

\begin{proof}
$(i)$ Suppose that $r_{i_0}-r_{i_0-1}\rightarrow0$ for some $i_0\in\{1,\cdots,k\}$, by the Hardy-Littlewood-Sobolev inequality, H\"{o}lder inequality and Sobolev inequality, we have
\begin{equation}\label{eq:3.2}
\begin{aligned}
\|w_{i_0}^{\mathbf{r}_k}\|_{i_0}^2&=\int_{\mathbb{R}^3}\int_{B_{i_0}^{\mathbf{r}_k}}\frac{|w^{\mathbf{r}_k}(x)|^p|w_{i_0}^{\mathbf{r}_k}(y)|^p}{|x-y|}\,dxdy\\
&\leq C\|w^{\mathbf{r}_k}\|_{\frac{6p}{5}}^p\|w_{i_0}^{\mathbf{r}_k}\|_{\frac{6p}{5}}^p\\
&\leq C\|w_{i_0}^{\mathbf{r}_k}\|_{i_0}^p|B_{i_0}^{\mathbf{r}_k}|^{\frac{5-p}{6}},
\end{aligned}
\end{equation}
which implies $\|w_{i_0}^{\mathbf{r}_k}\|_{i_0}\rightarrow+\infty$ as $r_{i_0}-r_{i_0-1}\rightarrow0$ since $\frac{5}{2}<p<5$. Thus, we derive from \eqref{eq:2.12} that
\[
\psi(\mathbf{r}_k)=E^{\mathbf{r}_k}(w_1^{\mathbf{r}_k},\cdots,w_{k+1}^{\mathbf{r}_k})\geq(\frac{1}{2}-\frac{1}{2p})\|w_{i_0}^{\mathbf{r}_k}\|_{i_0}^2\rightarrow\infty,
\]
Therefore, the first item holds.

$(ii)$ By the Strauss inequality \cite{S}, that is, for $u \in H^1_r(\mathbb{R}^3)$, there exists $C>0$, such that
\[
|u(x)|\leq C \frac{\|u\|}{|x|},\,\,a.e \quad{\rm in}\quad\mathbb{R}^3,
\]
we deduce as \eqref{eq:3.2} that
\[
\begin{aligned}
\|w_{k+1}^{\mathbf{r}_k}\|_{k+1}^2&=\int_{\mathbb{R}^3}\int_{B_{k+1}^{\mathbf{r}_k}}\frac{|w^{\mathbf{r}_k}(x)|^p|w_{k+1}^{\mathbf{r}_k}(y)|^p}{|x-y|}dxdy\\
                  &\leq C\big(\int_{B_{k+1}^{\mathbf{r}_k}}|w_{k+1}^{\mathbf{r}_k}(x)|^{\frac{6p}{5}}dx\big)^{\frac{5}{6}}\\
                   &\leq C\|w_{k+1}^{\mathbf{r}_k}\|_{k+1}^p\frac{5}{6p-15} r_k^{\frac{15-6p}{5}},
\end{aligned}
\]
that is,
\[
r_k^{\frac{6p-15}{5}}\leq C\frac{5}{6p-15}\|w_{k+1}^{\mathbf{r}_k}\|_{k+1}^{p-2}.
\]
Since $\frac{5}{2}<p<5$, we deduce that $\|w_{k+1}^{\mathbf{r}_k}\|_{k+1}\rightarrow+\infty$ as $r_k\rightarrow\infty$. Then, by \eqref{eq:2.12}, we obtain
\[
\psi(\mathbf{r}_k)=E^{\mathbf{r}_k}(w_1^{\mathbf{r}_k},\cdots,w_{k+1}^{\mathbf{r}_k})\geq(\frac{1}{2}-\frac{1}{2p})\|w_{k+1}^{\mathbf{r}_k}\|_{k+1}^2\rightarrow\infty,
\]
and the conclusion in $(ii)$ holds.

$(iii)$  Take a sequence $\{\mathbf{r}_k^n\}_{n=1}^\infty=\{(r_1^n,\cdots,r_k^n)\}\subseteq\mathbf{\Gamma}_k$ such that
\[
{\bf r}^n_k\rightarrow \tilde{\mathbf{r}}_k=(\tilde{r}_1,\cdots,\tilde{r}_k)\in {\bf \Gamma}_k.
\]
The assertion follows by showing
\begin{equation}\label{eq:3.3}
\psi(\tilde{\mathbf{r}}_k)\geq\limsup_{n\rightarrow\infty}\psi(\mathbf{r}_k^n),\,\,\,\psi(\tilde{\mathbf{r}}_k)\leq\limsup_{n\rightarrow\infty}\psi(\mathbf{r}_k^n).
\end{equation}

First, we prove $\psi(\tilde{\mathbf{r}}_k)\geq\limsup_{n\rightarrow\infty}\psi(\mathbf{r}_k^n)$. Defined $v_i^{\mathbf{r}_k^n}:[r_{i-1}^n,r_i^n]\rightarrow \mathbb{R}$ such that
\[
v_i^{\mathbf{r}_k^n}(t)=t_i^nw_i^{\tilde{\mathbf{r}}_k}\bigg(\frac{\tilde{r}_i-\tilde{r}_{i-1}}{r_i^n-r_{i-1}^n}(t-r_{i-1}^n)+\tilde{r}_{i-1}\bigg)
\]
for $i=1,\cdots,k$ and
\[
v_{k+1}^{\mathbf{r}_k^n}(t)=t_{k+1}^nw_{k+1}^{\tilde{\mathbf{r}}_k}\bigg(\frac{\tilde{r}_k}{r_k^n}t\bigg),
\]
where $r_0^n=0,\ r_{k+1}^n=\infty$ and each $(t_1^n,\cdots,t_{k+1}^n)$ is a unique $(k+1)$-tuple of positive real numbers such that $(v_1^{\mathbf{r}_k^n},\cdots,v_{k+1}^{\mathbf{r}_k^n})\in \mathcal{N}_k^{\mathbf{r}_k^n}$. By the definition of $(w_1^{\mathbf{r}^n_k},\cdots,w_{k+1}^{\mathbf{r}^n_k})$, we have
\[
E^{\mathbf{r}^n_k}(v_1^{\mathbf{r}_k^n},\cdots,v_{k+1}^{\mathbf{r}_k^n})\geq E^{\mathbf{r}^n_k}(w_1^{\mathbf{r}^n_k},\cdots,w_{k+1}^{\mathbf{r}^n_k})=\psi(\mathbf{r}^n_k)
\]
Therefore, for $n$ large enough, we have
\[
\|v_i^{\mathbf{r}_k^n}\|_{B_i^{\mathbf{r}_k^n}}^2=(t_i^n)^2\|w_i^{\tilde{\mathbf{r}}_k}\|_{B_i^{\tilde{\mathbf{r}}_k}}^2+o(1)
\]
and
\[
\begin{aligned}
&\int_{B_i^{\mathbf{r}_k^n}}\int_{B_i^{\mathbf{r}_k^n}}\frac{|v_i^{\mathbf{r}_k^n}(x)|^p|v_j^{\mathbf{r}_k^n}(y)|^p}{|x-y|}\,dxdy\\
&=(t_i^n)^p(t_j^n)^p\int_{B_i^{\tilde{\mathbf{r}}_k}}\int_{B_j^{\tilde{\mathbf{r}}_k}}\frac{|w_i^{\tilde{\mathbf{r}}_k}(x)|^p|w_j^{\tilde{\mathbf{r}}_k}(y)|^p}{|x-y|}\,dxdy+o(1).
\end{aligned}
\]
Since $(v_1^{\mathbf{r}_k^n},\cdots,v_{k+1}^{\mathbf{r}_k^n})\in \mathcal{N}_k^{\mathbf{r}_k^n}$, we have
\[
\begin{aligned}
&\|v_i^{\mathbf{r}_k^n}\|_{B_i^{\mathbf{r}_k^n}}^2
-\int_{B_i^{\mathbf{r}_k^n}}\int_{B_i^{\mathbf{r}_k^n}}\frac{|v_i^{\mathbf{r}_k^n}(x)|^p|v_i^{\mathbf{r}_k^n}(y)|^p}{|x-y|}\,dxdy\\
&-\sum_{j\neq i}^{k+1}\int_{B_i^{\mathbf{r}_k^n}}\int_{B_j^{\mathbf{r}_k^n}}\frac{|v_i^{\mathbf{r}_k^n}(x)|^p|v_j^{\mathbf{r}_k^n}(y)|^p}{|x-y|}\,dxdy=0
\end{aligned}
\]
for $i=1,\cdots,k+1$, which implies
\begin{equation}\label{eq:3.4}
\begin{aligned}
&(t_i^n)^2\|w_i^{\tilde{\mathbf{r}}_k}\|_{B_i^{\tilde{\mathbf{r}}_k}}^2
-(t_i^n)^{2p}\int_{B_i^{\tilde{\mathbf{r}}_k}}\int_{B_i^{\tilde{\mathbf{r}}_k}}\frac{|w_i^{\tilde{\mathbf{r}}_k}(x)|^p|w_i^{\tilde{\mathbf{r}}_k}(y)|^p}{|x-y|}\,dxdy\\
&-\sum_{j\neq i}^{k+1}(t_i^n)^p(t_j^n)^p\int_{B_i^{\tilde{\mathbf{r}}_k}}\int_{B_j^{\tilde{\mathbf{r}}_k}}\frac{|w_i^{\tilde{\mathbf{r}}_k}(x)|^p
|w_j^{\tilde{\mathbf{r}}_k}(y)|^p}{|x-y|}\,dxdy=o(1).
\end{aligned}
\end{equation}
Hence, the fact $(w_1^{\tilde{\mathbf{r}}_k},\cdots,w_{k+1}^{\tilde{\mathbf{r}}_k})\in \mathcal{N}_k^{\tilde{\mathbf{r}}_k}$, namely,
\begin{equation}\label{eq:3.5}
\begin{aligned}
&\|w_i^{\tilde{\mathbf{r}}_k}\|_{B_i^{\tilde{\mathbf{r}}_k}}^2
-\int_{B_i^{\tilde{\mathbf{r}}_k}}\int_{B_i^{\tilde{\mathbf{r}}_k}}\frac{|w_i^{\tilde{\mathbf{r}}_k}(x)|^p|w_i^{\tilde{\mathbf{r}}_k}(y)|^p}{|x-y|}\,dxdy\\
&-\sum_{j\neq i}^{k+1}\int_{B_i^{\tilde{\mathbf{r}}_k}}\int_{B_j^{\tilde{\mathbf{r}}_k}}\frac{|w_i^{\tilde{\mathbf{r}}_k}(x)|^p|w_j^{\tilde{\mathbf{r}}_k}(y)|^p}{|x-y|}\,dxdy=0
\end{aligned}
\end{equation}
and \eqref{eq:3.4} yield $\lim_{n\rightarrow\infty}t_i^n=1$ for all $i$. Consequently,
\[
\begin{aligned}
\psi(\tilde{\mathbf{r}}_k)&=E^{\tilde{\mathbf{r}}_k}(w_1^{\tilde{\mathbf{r}}_k},\cdots,w_{k+1}^{\tilde{\mathbf{r}}_k})
=\limsup_{n\rightarrow\infty}E^{\mathbf{r}^n_k}(v_1^{\mathbf{r}_k^n},\cdots,v_{k+1}^{\mathbf{r}_k^n})\\
&\geq\limsup_{n\rightarrow\infty}E^{\mathbf{r}^n_k}(w_1^{\mathbf{r}^n_k},\cdots,w_{k+1}^{\mathbf{r}^n_k})=\limsup_{n\rightarrow\infty}\psi(\mathbf{r}^n_k).
\end{aligned}
\]
This also implies that
\begin{equation}\label{eq:3.6}
\limsup_{n\rightarrow\infty}\|w_i^{\mathbf{r}_k^n}\|_{B_i^{\mathbf{r}_k^n}}^2<\infty,
\limsup_{n\rightarrow\infty}\int_{B_i^{\mathbf{r}_k^n}}\int_{B_i^{\mathbf{r}_k^n}}\frac{|w_i^{\mathbf{r}_k^n}(x)|^p|w_j^{\mathbf{r}_k^n}(y)|^p}{|x-y|}dxdy<\infty.
\end{equation}

Next, we turn to prove $\psi(\tilde{\mathbf{r}}_k)\leq\limsup_{n\rightarrow\infty}\psi(\mathbf{r}_k^n)$.

In the same way, we define functions $\bar{v}_i^{\mathbf{r}_k^n}:[\tilde{r}_{i-1},\tilde{r}_i]\rightarrow \mathbb{R}$ such that
\[
\bar{v}_i^{\mathbf{r}_k^n}(t)=s_i^nw_i^{\mathbf{r}_k^n}\bigg(\frac{r_i^n-r_{i-1}^n}{\tilde{r}_i-\tilde{r}_{i-1}}(t-\tilde{r}_{i-1})+r_{i-1}^n\bigg)
\]
for $i=1,\cdots,k$ and
\[
\bar{v}_{k+1}^{\mathbf{r}_k^n}(t)=s_{k+1}^nw_{k+1}^{\mathbf{r}_k^n}\bigg(\frac{r_k^n}{\tilde{r}_k}t\bigg),
\]
where $r_0^n=0,\ r_{k+1}^n=\infty$ and each $(s_1^n,\cdots,s_{k+1}^n)$ is a unique $(k+1)$-tuple of positive real numbers such that $(\bar{v}_1^{\mathbf{r}_k^n},\cdots,\bar{v}_{k+1}^{\mathbf{r}_k^n})\in \mathcal{N}_k^{\tilde{\mathbf{r}}_k}$. Then, by the definition of $(w_1^{\tilde{\mathbf{r}}_k},\cdots,w_{k+1}^{\tilde{\mathbf{r}}_k})$, we have
\[
E^{\mathbf{r}^n_k}(\bar{v}_1^{\mathbf{r}_k^n},\cdots,\bar{v}_{k+1}^{\mathbf{r}_k^n})\geq E^{\tilde{\mathbf{r}}_k}(w_1^{\tilde{\mathbf{r}}_k},\cdots,w_{k+1}^{\tilde{\mathbf{r}}_k})=\psi({\tilde{\mathbf{r}}_k}).
\]
Similarly, we may derive that
\begin{equation}\label{eq:3.7}
\begin{aligned}
&(s_i^n)^2\|w_i^{\mathbf{r}_k^n}\|_{B_i^{\mathbf{r}_k^n}}^2-(s_i^n)^{2p}\int_{B_i^{\mathbf{r}_k^n}}\int_{B_i^{\mathbf{r}_k^n}}
\frac{|w_i^{\mathbf{r}_k^n}(x)|^p|w_i^{\mathbf{r}_k^n}(y)|^p}{|x-y|}\,dxdy\\
&-\sum_{j\neq i}^{k+1}(s_i^n)^p(s_j^n)^p\int_{B_i^{\mathbf{r}_k^n}}\int_{B_j^{\mathbf{r}_k^n}}\frac{|w_i^{\mathbf{r}_k^n}(x)|^p|w_j^{\mathbf{r}_k^n}(y)|^p}{|x-y|}\,dxdy=o(1)
\end{aligned}
\end{equation}
and
\begin{equation}\label{eq:3.8}
\begin{aligned}
&\|w_i^{\mathbf{r}_k^n}\|_{B_i^{\mathbf{r}_k^n}}^2
-\int_{B_i^{\mathbf{r}_k^n}}\int_{B_i^{\mathbf{r}_k^n}}\frac{|w_i^{\mathbf{r}_k^n}(x)|^p|w_i^{\mathbf{r}_k^n}(y)|^p}{|x-y|}\,dxdy\\
&-\sum_{j\neq i}^{k+1}\int_{B_i^{\mathbf{r}_k^n}}\int_{B_j^{\mathbf{r}_k^n}}\frac{|w_i^{\mathbf{r}_k^n}(x)|^p|w_j^{\mathbf{r}_k^n}(y)|^p}{|x-y|}\,dxdy=0
\end{aligned}
\end{equation}
for each $i=1,\cdots,k+1$. We deduce from \eqref{eq:3.7} and  \eqref{eq:3.8}  that $\lim_{n\rightarrow\infty}s_i^n=1$ for all $i$. Therefore,
\[
\begin{aligned}
\psi({\tilde{\mathbf{r}}_k})&=E^{\tilde{\mathbf{r}}_k}(w_1^{\tilde{\mathbf{r}}_k},\cdots,w_{k+1}^{\tilde{\mathbf{r}}_k})
\leq\liminf_{n\rightarrow\infty}E^{\mathbf{r}^n_k}(\bar{v}_1^{\mathbf{r}_k^n},\cdots,\bar{v}_{k+1}^{\mathbf{r}_k^n})\\
&=\liminf_{n\rightarrow\infty}E^{\mathbf{r}^n_k}(w_1^{\mathbf{r}^n_k},\cdots,w_{k+1}^{\mathbf{r}^n_k})=\liminf_{n\rightarrow\infty}\psi(\mathbf{r}^n_k).
\end{aligned}
\]
This completes the proof of $(iii)$.

As a result, we infer from $(i)-(iii)$ that there is a minimum point $\bar{\mathbf{r}}_k=(\bar{r}_1,\cdots,\bar{r}_k)\in\mathbf{\Gamma}_k$ of $\psi$.

\end{proof}

Finally, we show that the solution $(w_1^{\bar{\mathbf{r}}_k},\cdots ,w_{k+1}^{\bar{\mathbf{r}}_k})$ of $(P_i)$, corresponding to the point $\bar{\mathbf{r}}_k=(\bar{r}_1,\cdots,\bar{r}_k)\in\mathbf{\Gamma}_k$ which we found in the previous lemma, is the exact element which gives the solution of $(P)$ with desired sign changing property.

\medskip
{\bf Proof of Theorem \ref{thm:1.1}} Suppose on the contrary that $\sum_{i=1}^{k+1}w_i^{\bar{\mathbf{r}}_k}$ is not a solution of $(P)$, there would exist $l\in\{1,\cdots,k\}$ such that
\begin{equation}\label{eq:3.9}
w_-=\lim_{t\rightarrow \bar{r}_l^-}\frac{dw_l^{\bar{\mathbf{r}}_k}(t)}{dt}\neq\lim_{t\rightarrow \bar{r}_l^+}\frac{dw_{l+1}^{\bar{\mathbf{r}}_k}(t)}{dt}=w_+.
\end{equation}

Denote $w_l(t) = w_l^{\bar{\mathbf{r}}_k}(t)$ and $w_{l+1}(t) = w_{l+1}^{\bar{\mathbf{r}}_k}(t)$. Fix a small positive number $\delta$ and set
\begin{equation*}
\bar{y}(t)=\left\{
\begin{aligned}
&w_l(t),&{\rm if}\ t\in(\bar{r}_{l-1},\bar{r}_l-\delta), \\
&w_l(\bar{r}_l-\delta)
+\frac{w_{l+1}(\bar{r}_l+\delta)-w_l(\bar{r}_l-\delta)}{2\delta}(t-\bar{r}_l+\delta),&{\rm if}\ t\in(\bar{r}_l-\delta,\bar{r}_l+\delta), \\
&w_{l+1}(t),&{\rm if}\ t\in(\bar{r}_l+\delta,\bar{r}_{l+1}). \\
\end{aligned} \right.
\end{equation*}
There exists a unique ${\bar s}_l\in(\bar{r}_{l-1}-\delta,\bar r_{l+1}+\delta)$ such that
\[
{\bar y}(t)|_{t={\bar s}_l}=0
\]
since $\bar{y}(\bar{r}_{l-1}-\delta)\bar{y}(\bar{r}_l+\delta)<0$. Define a $(k+1)$-tuple of functions $(\bar{z}_1,\cdots,\bar{z}_{k+1})$ as follows.
\[
\left\{\begin{array}{c@{\quad }l}
\bar{z}_l(t)=\bar{y}(t),&{\rm for} \ t\in(\bar{r}_{l-1},\bar{s}_l), \\
\bar{z}_{l+1}(t)=\bar{y}(t),&{\rm for} \ t\in(\bar{s}_l,\bar{r}_{l+1}), \\
\bar{z}_i(t)=w_i^{\bar{\mathbf{r}}_k}(t),&{\rm for} \ t\in(\bar{r}_{i-1},\bar{r}_{i})\ {\rm if}\ i\neq l,l+1. \\
\end{array} \right.\
\]
By Lemma \ref{lem:2.1},  there is a unique $(k+1)$-tuple $(\hat{t}_1,\cdots,\hat{t}_{k+1})\in(\mathbb{R}_{>0})^{k+1}$ such that
\[
(z_1^{\bar{\mathbf{s}}},\cdots,z_{k+1}^{\bar{\mathbf{s}}}):=(\hat{t}_1\bar{z}_1,\cdots,\hat{t}_{k+1}\bar{z}_{k+1})\in \mathcal{N}_k^{\bar{\mathbf{s}}}
\]
with $\bar{\mathbf{s}}=(\bar{r}_1,\cdots,\bar{r}_{l-1},\bar{s},\bar{r}_{l+1},\cdots,\bar{r}_k)$. On
the other hand, we can verify that
\begin{equation}\label{eq:3.10}
(\hat{t}_1,\cdots,\hat{t}_{k+1})\rightarrow(1,\cdots,1)
\end{equation}
as $\delta\rightarrow0$. Let $W(t):=\sum_{i=1}^{k+1}w_i^{\bar{\mathbf{r}}_k}(t)\in H^1_r(\mathbb{R}^3)$ and $Z(t):=\sum_{i=1}^{k+1}z_i^{\bar{\mathbf{s}}}(t)\in H^1_r(\mathbb{R}^3)$.
Then
\begin{equation}\label{eq:3.11}
E(W)=E^{\bar{\mathbf{r}}_k}(w_1^{\bar{\mathbf{r}}_k},\cdots,w_{k+1}^{\bar{\mathbf{r}}_k})\leq E^{\bar{\mathbf{s}}}(z_1^{\bar{\mathbf{s}}},\cdots,z_{k+1}^{\bar{\mathbf{s}}})=E(Z).
\end{equation}

On the other hand, for any $f\in H^1_r(\mathbb{R}^3)$, the solution $\varphi$ of $-\Delta\varphi=f$ is radial and it can be expressed as
\[
\varphi(t)=\frac{1}{t}\int_{0}^{\infty}f(s)s\min\{s,t\}\,ds
\]
for $t>0$. Therefore, $W$ satisfies
\begin{equation}\label{eq:3.12}
\int_0^\infty t^2(W'^{2}+W^2)dt=\int_0^\infty\int_0^\infty|W(s)|^p|W(t)|^p st\min\{s,t\}\,dsdt
\end{equation}
and
\begin{equation}\label{eq:3.13}
\begin{split}
E(W)&=\frac{1}{2}\int_0^\infty (W'^{2}+W^2)t^2dt\\
&\quad-\frac{1}{2p}\int_0^\infty\int_0^\infty|W(s)|^p|W(t)|^pst \min\{s,t\}\,dsdt\\
&=\big(\frac{1}{2}-\frac{1}{2p}\big)\int_0^\infty\int_0^\infty|W(s)|^p|W(t)|^pst\min\{s,t\}\,dsdt.
\end{split}
\end{equation}
We deduce from
\[
w_-= \lim_{\delta\to 0}\frac{W(\bar r_l-\delta)-W(\bar r_l)}{-\delta}
\]
that
\begin{equation}\label{eq:3.14}
W(\bar r_l-\delta)=-\delta w_-+o(\delta).
\end{equation}
Since $W$ satisfies
\[
-\big(t^2W'\big)'+t^2W=\int_0^\infty |W(s)|^pst\min\{s,t\}\,ds|W|^{p-2}W(t)
\]
for $\bar{r}_l-\delta\leq t\leq \bar{r}_l$, and $W(\bar{r}_l)= 0$,   thereby $\big(t^2W'\big)'(\bar{r}_l)=0$, we obtain
\begin{equation}\label{eq:3.15}
 (\bar{r}_l-\delta)^2W'(\bar{r}_l-\delta)=\bar{r}_l^2w_-+o(\delta).
\end{equation}
We write
\[
\begin{split}
E(Z)&=\frac{1}{2}\int_0^\infty (Z'^{2}+Z^2)t^2dt-\frac{1}{2p}\int_0^\infty\int_0^\infty|Z(s)|^p|Z(t)|^pst \min\{s,t\}\,dsdt\\
&=\frac{1}{2}\bigg(\int_0^{\bar{r}_l-\delta}+\int_{\bar{r}_l+\delta}^\infty\bigg)(Z'^{2}+Z^2)t^2dt+\frac{1}{2}\int_{\bar{r}_l-\delta}^{\bar{r}_l+\delta}(Z'^{2}+Z^2)t^2dt\\
&\quad-\frac{1}{2p}\int_0^\infty\int_0^\infty|Z(s)|^p|Z(t)|^pst\min\{s,t\}\,dsdt.
\end{split}
\]
By \eqref{eq:3.10}, we see that
\[
\int_0^{\bar{r}_l-\delta}(Z'^{2}+Z^2)t^2\,dt=\int_0^{\bar{r}_l-\delta}(W'^{2}+W^2)t^2\,dt+o(\delta).
\]
Integrating by part and using \eqref{eq:3.14} and \eqref{eq:3.15}, we obtain that
\[
\begin{split}
&\int_0^{\bar{r}_l-\delta}(W'^{2}+W^2)t^2\,dt+o(\delta)\\
&=W'(\bar{r}_l-\delta)W(\bar{r}_l-\delta)(\bar{r}_l-\delta)^2\\
&+\int_0^{\bar{r}_l-\delta}\int_0^\infty|W(s)|^p|W(t)|^pst \min\{s,t\}\, dsdt\\
&=-\delta (w_-)^{2}\bar{r}_l^2+\int_0^{\bar{r}_l-\delta}\int_0^\infty|W(s)|^p|W(t)|^pst\min\{s,t\}\,dsdt.
\end{split}
\]
Thus,
\begin{equation}\label{eq:3.16}
\begin{split}
&\int_0^{\bar{r}_l-\delta}(Z'^{2}+Z^2)t^2\,dt\\
&=-\delta (w_-)^{2}\bar{r}_l^2+\int_0^{\bar{r}_l-\delta}\int_0^\infty|W(s)|^p|W(t)|^pst\min\{s,t\}\,dsdt+o(\delta).
\end{split}
\end{equation}
In the same way,
\begin{equation}\label{eq:3.17}
\begin{split}
&\int_{\bar{r}_l+\delta}^\infty(Z'^{2}+Z^2)t^2\,dt\\
=&-\delta (w_+)^{2}\bar{r}_l^2+\int_{\bar{r}_l+\delta}^{\infty}\int_0^\infty|W(s)|^p|W(t)|^pst\min\{s,t\}dsdt+o(\delta).
\end{split}
\end{equation}
It is readily to verify that
\begin{equation}\label{eq:3.18}
\int_{\bar{r}_l-\delta}^{\bar{r}_l+\delta}Z'^{2}t^2\,dt=\frac{1}{2}\bar{r}_l^2(w_++w_-)^2\delta+o(\delta)
\end{equation}
and
\begin{equation}\label{eq:3.19}
\int_{\bar{r}_l-\delta}^{\bar{r}_l+\delta}Z^2t^2\,dt=o(\delta).
\end{equation}
From \eqref{eq:3.16}-\eqref{eq:3.19}, we obtain
\begin{equation}\label{eq:3.20}
\begin{aligned}
E(Z)=&-\frac{\delta}{2}(w_-)^{2}\bar{r}_l^2+\frac{1}{2}\int_0^{\bar{r}_l-\delta}\int_0^\infty|W(s)|^p|W(t)|^pst\min\{s,t\}\,dsdt\\
&-\frac{\delta}{2}(w_+)^{2}\bar{r}_l^2+\frac{1}{2}\int_{\bar{r}_l+\delta}^{\infty}\int_0^\infty|W(s)|^p|W(t)|^p st\min\{s,t\}\,dsdt\\
&+\frac{\delta}{4}\bar{r}_l^2(w_++w_-)^2-\frac{1}{2p}\int_0^\infty\int_0^\infty|W(s)|^p|W(t)|^pst\min\{s,t\}\,dsdt\\
&+o(\delta).
\end{aligned}
\end{equation}
Consequently,
\begin{equation}\label{eq:3.21}
\begin{aligned}
&E(Z)-E(W)\\
&=-\frac{\delta}{4}\bar{r}_l^2(w_+-w_-)^2\\
&\quad+\frac{1}{2}\bigg(\int_0^{\bar{r}_l-\delta}+\int_{\bar{r}_l+\delta}^\infty\bigg)\int_0^\infty|W(s)|^p|W(t)|^pst\min\{s,t\}\,dsdt\\
&\quad-\frac{1}{2p}\int_0^\infty\int_0^\infty|W(s)|^p|W(t)|^pst\min\{s,t\}dsdt\\
&\quad-\big(\frac{1}{2}-\frac{1}{2p}\big)\int_0^\infty\int_0^\infty|W(s)|^p|W(t)|^pst\min\{s,t\}dsdt +o(\delta)\\
&=-\frac{\delta}{4}\bar{r}_l^2(w_+-w_-)^2-\frac{1}{2}\int_{\bar{r}_l-\delta}^{\bar{r}_l+\delta}\int_0^\infty|W(s)|^p|W(t)|^pst\min\{s,t\}\,dsdt+o(\delta).\\
\end{aligned}
\end{equation}
This together with the fact
\[
\int_{\bar{r}_l-\delta}^{\bar{r}_l+\delta}\int_0^\infty|W(s)|^p|W(t)|^pst\min\{s,t\}dsdt=o(\delta)
\]
yields
\[
E(Z)-E(W)=-\frac{\delta}{4}\bar{r}_l^2(w_+-w_-)^2+o(\delta)<0
\]
if $\delta>0$ sufficiently small, which contradicts \eqref{eq:3.11}. The proof is complete. $\Box$

\bigskip

\appendix

\section {Non-singularity of matrices}

\bigskip
We show in this section that the matrices $M$ and $N$ defined in \eqref{eq:2.5a} and \eqref{eq:2.10} respectively are nonsingular. For $f, g\in L^1_{loc}(\mathbb{R}^3)$, we recall that the Coulomb energy is defined in \cite{LL} by
\[
D_N(f,g) = \int_{\mathbb{R}^N}\int_{\mathbb{R}^N}f(x)g(y)|x-y|^{2-N}\,dxdy.
\]
It is proved in Theorem 9.8 of  \cite{LL} the following result.

\begin{Lemma}\label{lem:A.1}
(\cite{LL}, Theorem 9.8)
Let $N\geq1$ and $f,g\in L^{\frac{2N}{N+2}}$, then
\[
|D_N(f,g)|^2\leq D_N(f,f)D_N(g,g),
\]
with equality for $g\neq0$ if and only if $f=Cg$ for some constant $C$.
\end{Lemma}

\medskip

Denote $D(f,g)= D_3(f,g)$. Let
$$
\mathcal{A}(\mathbb{R}^3):=\big\{f\in L_{loc}^1(\mathbb{R}^3):D(f,f)<\infty\big\}.
$$

\begin{Lemma}\label{lem:A.2}
$\mathcal{A}(\mathbb{R}^3)$ is  a linear subspace of $L_{loc}^1(\mathbb{R}^3)$ with the inner product $D(f,f)$.
\end{Lemma}
\begin{proof} By Lemma \ref{lem:A.1}, for any $f, g \in \mathcal{A}(\mathbb{R}^3)$, we have
\[
D(f+g,f+g)\leq D(f,f)+D(g,g)+2\sqrt{D(f,f)D(g,g)}.
\]
It is then readily to verify that $\mathcal{A}(\mathbb{R}^3)$ is a linear subspace of $L_{loc}^1(\mathbb{R}^3)$.
It is also standard to see that $D(f,g)$ is an inner product in  $\mathcal{A}(\mathbb{R}^3)$.
\end{proof}

Now, we show that the matrices $M$ and $N$ defined in \eqref{eq:2.5a} and \eqref{eq:2.10} respectively are nonsingular.
We only prove the matrix $N$ is nonsingular, since for the matrix $M$, the proof is similar.
\begin{Lemma}\label{lem:A}
The matrix $N$ defined in \eqref{eq:2.10} is nonsingular.
\end{Lemma}
\begin{proof} Denote $v_i:=|u_i(x)|^p$. Then $v_i\in \mathcal{A}(\mathbb{R}^3)$, for $i=1,\cdots,k+1$.
Apparently, $v_1,\cdots,v_{k+1}$ are linear independent. Let
\[
L=span\{v_1,\cdots,v_{k+1}\}.
\]
So $L$ is a subspace of $\mathcal{A}(\mathbb{R}^3)$. Denote by $\{e_1,..., e_{k+1}\}$ the orthogonal basis obtained from $\{v_1,\cdots,v_{k+1}\}$ by the Gram-Schmidt Orthogonalization procedure. We may assume  $v_i=\Sigma^{k+1}_{j=1} a_{ij} e_j$ for $i=1,...,k+1$. Then, the matrix
$$
A_{k+1}=\begin{pmatrix}
a_{11} & a_{12} & \cdots & a_{1(k+1)}\\
\vdots   & \vdots & \ddots& \vdots\\
a_{(k+1)1} & a_{(k+1)2} & \cdots & a_{(k+1)(k+1)}
\end{pmatrix}
$$
is invertible.

Denote $D_{ij}= v_iv_j = D(v_i,v_j)$ for $i,j=1,\cdots,k+1$. The matrix $(D_{ij})_{(k+1)\times (k+1)}$ can be written as
\[(
D_{ij})_{(k+1)\times (k+1)}
=\begin{pmatrix}
	v_1\\
	\vdots\\
	v_{k+1}
	\end{pmatrix}
	\begin{pmatrix}
	v_1&v_2 & \cdots & v_{k+1}
	\end{pmatrix}.
\]
Using the fact that $v_i=\Sigma^{k+1}_{j=1} a_{ij} e_j$ for $i=1,...,k+1$ and $(e_1,\cdots, e_{k+1})$ is a orthogonal basis, we deduce
\begin{align*}
&\begin{pmatrix}
	v_1\\
	\vdots\\
	v_{k+1}
	\end{pmatrix}
	\begin{pmatrix}
	v_1&v_2 & \cdots & v_{k+1}
	\end{pmatrix}\\
    &=\begin{pmatrix}
	a_{11} & a_{12} & \cdots & a_{1(k+1)}\\
	\vdots & \vdots & \ddots & \vdots\\
	a_{(k+1)1} & a_{(k+1)2} & \cdots & a_{(k+1)(k+1)}
	\end{pmatrix}\begin{pmatrix}
	a_{11} & a_{21} & \cdots & a_{(k+1)1}\\
	\vdots & \vdots & \ddots & \vdots\\
	a_{1(k+1)} & a_{2(k+1)} & \cdots & a_{(k+1)(k+1)}
	\end{pmatrix}.
\end{align*}
Therefore,
\[
(D_{ij})_{(k+1)\times (k+1)}=A_{k+1} A^T_{k+1}
\]
Since $A_{k+1}$ is invertible, the matrix $(D_{ij})_{(k+1)\times (k+1)}$ is positive definite.

Let $d_i=\|u_i\|^2_i$, $i=1,\cdots,k+1$. It is obvious that
\begin{equation}\label{eq:2.17}
{\rm det}\ N=(-1)^{k+1}{\rm det}\ \widetilde{N},
\end{equation}
where the matrix $\widetilde{N}$ is given by
\begin{align*}
	\widetilde{N}&=\begin{pmatrix}
	pD_{11}+(p-2)d_1 & pD_{12} & \cdots & pD_{1(k+1)}\\
	pD_{21} & pD_{22}+(p-2)d_2& \cdots & pD_{2(k+1)}\\
	\vdots & \vdots & \ddots & \vdots\\
	pD_{(k+1)1} & pD_{(k+1)2}& \cdots & pD_{(k+1)(k+1)}+(p-2)d_{k+1}
	\end{pmatrix}\\
	&=p(D_{ij})_{(k+1)\times (k+1)}+
    (p-2)\begin{pmatrix}
	d_1 &  & & \\
	  & d_2  & & \\
	 &   & \ddots &  \\
	  &   &   & d_{k+1}
	\end{pmatrix}.
\end{align*}
So 	$\widetilde{N}$ is positive definite if $\frac{5}{2}<p<5$ since $d_i>0$ for all $i$ and $(D_{ij})_{(k+1)\times (k+1)}$ is positive definite.
The conclusion then follows.
\end{proof}

\bigskip
\vspace{2mm}
\noindent{\bf Acknowledgment} This work is supported by NNSF of China, No:11271170 and 11371254; GAN
PO 555 program of Jiangxi.

\end{document}